%% file: NPG_PotentialGame.tex
\documentclass{article}

\usepackage{geometry}
\usepackage{graphics} 
\usepackage{epsfig} 
\usepackage{amsmath} 
\usepackage{amssymb}  
\usepackage{amsthm}

\usepackage[normalem]{ulem}
\usepackage[linesnumbered,ruled,vlined]{algorithm2e}

\usepackage[unicode=true, bookmarks=false,
 breaklinks=false,pdfborder={0 0 1}]
 {hyperref}
\hypersetup{
 colorlinks,citecolor=blue,filecolor=blue,linkcolor=blue,urlcolor=blue}

\usepackage{bm}
\usepackage[normalem]{ulem}
\usepackage{xcolor}

\allowdisplaybreaks

\definecolor{yxc}{RGB}{255,0,0}
\definecolor{yc}{RGB}{190,0,255} 
\definecolor{ytw}{RGB}{255,0,127}
\definecolor{dacong}{RGB}{10,103,68}

\newtheorem{theorem}{Theorem}[section]

\newtheorem{lemma}[theorem]{Lemma}
\newtheorem{corollary}[theorem]{Corollary}
\newtheorem{definition}[theorem]{Definition}

\allowdisplaybreaks

\title{ 
Independent Natural Policy Gradient Methods for Potential Games: Finite-time Global Convergence with Entropy Regularization  
}

\author{Shicong Cen{\footnote{The first two authors contributed equally.}
\thanks{S. Cen and Y. Chi are with the Department of Electrical and Computer Engineering, Carnegie Mellon University; emails: {\tt\small \{shicongc,yuejiec\}@andrew.cmu.edu}.}}\\
Carnegie Mellon University \\
\and Fan Chen{\footnotemark[1] \thanks{F. Chen is with the Department of Mathematics at Peking University; email: {\tt\small chern@pku.edu.cn}.}}\\
Peking University \\
\and Yuejie Chi\footnotemark[2]\\
Carnegie Mellon University
}

\date{April 13, 2022; \quad Revised \today}

\input{command}

\begin{document}

\maketitle
\thispagestyle{empty}
\pagestyle{empty}

\input{abstract}

\input{intro}

\input{potential_game}

\input{entropy_potential_game}

\input{pf-matrix}

\input{numerical}

\input{conclusions}

       
\section*{Acknowledgments}
The work of S.~Cen and Y.~Chi is supported in part by the grants ONR N00014-19-1-2404, ARO W911NF-18-1-0303, NSF CCF-1901199, CCF-2007911, CCF-2106778 and CNS-2148212. S.~Cen is also gratefully supported by Wei Shen and Xuehong Zhang Presidential Fellowship, and Nicholas Minnici Dean's Graduate Fellowship in Electrical and Computer Engineering at Carnegie Mellon University.
F.~Chen is supported by the Elite Undergraduate Training Program of School of Mathematical Sciences at Peking University.

\bibliographystyle{alphaabbr}
\bibliography{bibfileGame,bibfileRL}

\appendix

\input{pf-lemma}

\end{document}

%% file: command.tex
\newcommand{\ex}[2]{\mathbb{E}_{#1}\left[#2\right]}
\newcommand{\exlim}[2]{\mathop\mathbb{E}\limits_{#1}\left[#2\right]}

\newcommand{\prn}[1]{\left({#1}\right)} 
\newcommand{\prnbig}[1]{\big({#1}\big)} 

\newcommand{\KL}[2]{\mathsf{KL}\prnbig{{#1}\,\|\,{#2}}}

\newcommand{\innprod}[1]{\big\langle{#1} \big\rangle} 

\newcommand{\norm}[1]{\left\|{#1}\right\|} 

\newcommand{\brk}[1]{\left[{#1}\right]} 

\newcommand{\cA}{\mathcal{A}}
\newcommand{\cH}{\mathcal{H}}

\definecolor{dacong}{RGB}{10,103,68}

%% file: abstract.tex
\begin{abstract}
    A major challenge in multi-agent systems is that the system complexity grows dramatically with the number of agents as well as the size of their action spaces, which is typical in real world scenarios such as autonomous vehicles, robotic teams, network routing, etc. It is hence in imminent need to design decentralized or independent algorithms where the update of each agent is only based on their local observations without the need of introducing complex communication/coordination mechanisms.  

    In this work, we study the finite-time convergence of independent entropy-regularized natural policy gradient (NPG) methods for potential games, where the difference in an agent's utility function due to unilateral deviation matches exactly that of a common potential function. The proposed entropy-regularized NPG method enables each agent to deploy symmetric, decentralized, and multiplicative updates according to its own payoff. We show that the proposed method converges to the quantal response equilibrium (QRE)---the equilibrium to the entropy-regularized game---at a sublinear rate, which is independent of the size of the action space and grows at most sublinearly with the number of agents. Appealingly, the convergence rate further becomes {\em independent} with the number of agents for the important special case of identical-interest games, leading to the first method that converges at a {\em dimension-free} rate. Our approach can be used as a smoothing technique to find an approximate Nash equilibrium (NE) of the unregularized problem without assuming that stationary policies are isolated. 
\end{abstract}

%% file: intro.tex
\section{Introduction}



Reinforcement learning (RL) has garnered a growing amount of interest in recent years, due to its excellent empirical performance in a wide variety of applications, such as Go \cite{silver2016mastering}, motor control \cite{levine2016end}, chip design \cite{mirhoseini2021graph}, and so on. There is a growing interest to applying RL techniques such as Q-learning and policy gradient methods to the setting with multiple agents, i.e., multi-agent reinforcement learning (MARL) problems. 

While it seems appealing to apply single-agent RL methods to each agent in a multi-agent system in a straightforward fashion, this approach neglects non-stationarity of the environment due to the presence of other agents, and thus lacks theoretical support in general.
The complication has given rise to the paradigm of \textit{centralized training with decentralized execution} (CTDE) \cite{lowe2017multi}, where the policies are first obtained through training with a centralized controller with access to all agents' observations and then disseminated to each agent for execution. However, this approach falls short of adapting to changes in the environment without retraining and raises privacy concerns as well. 
It is hence of great interest to understand and design more versatile \textit{independent learning} algorithms that only depend on the agents' local observations, require minimal coordination between agents, and provably converge.

In this work, we focus on independent learning algorithms for potential games \cite{monderer1996potential}, an important class of games that admit a potential function to capture the differences in each agent's utility function induced by unilateral deviations. In particular, the analysis established in this work is tailored to potential games in their most basic setting, i.e., static potential games, an important stepping stone to the more general Markov setting. Despite its simple formulation and decades-long research, however, the computational underpinnings of such problems are still far from mature, especially when it comes to finding the Nash equilibrium (NE) of potential games in a decentralized manner. While several recent works have made significant breakthroughs by achieving logarithmic regrets with independent learning dynamics \cite{daskalakis2021nearoptimal,anagnostides2021near}, these results only guarantee convergence to coarse correlated equilibrium or correlated equilibrium, which are arguably much weaker equilibrium concepts than NE and hence do not lead to an approximate NE solution.

\subsection{Our contributions}
\label{sec:contribution}
We seek to find the quantal response equilibrium (QRE) \cite{mckelvey1995quantal}, the prototypical extension of NE for games with bounded rationality \cite{selten1989evolution}, where each agent runs independent natural policy gradient (NPG) methods \cite{kakade2001natural}  involving symmetric, decentralized, and multiplicative updates according to its own payoff. This amounts to solving a potential game with entropy regularization, whose algorithmic role has been studied in the setting of single-agent RL \cite{mei2020global,cen2020fast} as well as two-player zero-sum games \cite{cen2021fast}, but yet to be explored in more general settings.  
Our contributions are summarized below.
\begin{itemize} 
    \item \textbf{Finite-time global convergence of independent entropy-regularized NPG methods.} We show that independent entropy-regularized NPG methods provably converge to the QRE of a potential game, and it takes no more than
    \[
        \mathcal{\mathcal{O}}\prn{\frac{\min \{\sqrt{N}, \Phi_{\max} \}\Phi_{\max}}{\tau^2\epsilon^2}}
    \]
    iterations to find an $\epsilon$-optimal QRE (to be defined precisely). Here, $N$ stands for the number of agents, $\tau > 0$ the entropy regularization parameter, and $\Phi_{\max} > 0$ the maximum value of the potential function. 
    \vspace{1ex}
    \item \textbf{Finite-time global convergence to $\epsilon$-NE without isolation assumption.}
    By setting the entropy regularization parameter $\tau$ sufficiently small, the result translates to finding an approximate NE with non-asymptotic convergence guarantees, thereby obviating the additional assumption in prior literature \cite{fox2021independent, palaiopanos2017multiplicative,zhang2022effect} that requires the set of stationary policies to be isolated. Specifically, it takes no more than
    \[
        \widetilde{\mathcal{O}}\prn{\frac{\min\{\sqrt{N}, \Phi_{\max}\} \Phi_{\max}}{\epsilon^4}}
    \]
    iterations to find an $\epsilon$-NE for the unregularized potential game, where $\widetilde{\mathcal{O}}$ hides logarithmic dependencies.
\end{itemize}
These rates give the first set of iteration complexities---to the best of our knowledge---that are independent of the size of the action spaces, up to logarithmic factors. In addition, the iteration complexities exhibit a sublinear dependency with the number of agents, outperforming existing NE-finding algorithms whose complexities depend at least linearly with the number of agents. Even more appealingly, when interpreting our convergence rates for the important special case of identical-interest games with bounded payoffs \cite{monderer1996fictitious}, they further become independent with the number of agents, leading to the {\em first} method that achieves a {\em dimension-free} convergence rate of $ \widetilde{\mathcal{O}}\prn{1 /\epsilon^4} $ to find an $\epsilon$-NE.

\subsection{Related works}

We review some related works, focusing on the theoretical advances on policy gradient methods and independent learning in games.

\paragraph{Global convergence of policy gradient methods.} Only recently theoretical understandings on the global convergence of policy gradient (PG) methods have emerged, mostly in the single-agent setting, including but not limited to \cite{fazel2018global,agarwal2019optimality,mei2020global,cen2020fast,bhandari2019global,liu2020improved,li2021softmax,mei2020escaping,wang2019neural,xiao2022convergence}. In addition, several works developed finite-time guarantees of independent PG methods for zero-sum two-player Markov games \cite{daskalakis2020independent,wei2021last,zhao2021provably,cen2021fast} in the competitive MARL setting. A recent line work has been successful in extending PG methods with direct parameterization to Markov potential games \cite{zhang2021gradient, leonardos2021global, ding2022independent, mao2022improving}. In addition, \cite{zhang2022effect} studies the finite-time convergence rate of PG methods with softmax parameterization for Markov potential games. Given that NPG methods often have better finite-time convergence rates than vanilla PG methods in the single-agent setting, our work focuses on the understanding of NPG methods for potential games.  

\paragraph{Fast convergence of natural policy gradient methods with entropy regularization.} Entropy regularization as a de facto trick to promote exploration in RL \cite{haarnoja2018soft} and has been shown to provably accelerate convergence of policy gradient methods for single-agent RL \cite{mei2020global, cen2020fast, cen2021fast, zhan2021policy, lan2021policy}. In particular, combining entropy regularization with NPG methods leads to fast linear convergence at a desirable dimension-free rate \cite{cen2020fast,zhan2021policy, lan2021policy}, which continues to hold in two-player zero-sum games \cite{cen2021fast}. Extending such results to potential games, however, is non-trivial, due to the non-uniqueness of NE even with regularization. \cite{fox2021independent, palaiopanos2017multiplicative,zhang2022effect} established asymptotic convergence of independent NPG methods for Markov potential games with an additional assumption that requires the set of stationary policies to be isolated. \cite{heliou2017learning} demonstrated asymptotic convergence of NPG with diminishing step sizes for potential games in the bandit feedback setting. In addition, \cite{zhang2022effect} proposed to use a log-barrier regularization along with NPG to sidestep the isolation assumption and achieved the same iteration complexity as that of PG methods with direct parameterization. In contrast, we consider NPG with entropy regularization, which achieves a convergence rate that has better dependencies with the size of the action spaces and the number of agents.


\paragraph{Independent learning in general-sum games.} Considerable progress has been made towards understanding independent learning dynamics in general-sum games \cite{daskalakis2021nearoptimal,anagnostides2021near} and general-sum Markov games (also known as stochastic games) \cite{song2021can,jin2021v,mao2022provably} by establishing non-asymptotic convergence to correlated equilibrium and coarse correlated equilibrium. However, such successes do not directly extend to potential games where NE is of interest. Specialized analysis for potential games is thus needed as finding approximate NE in a two-player game can be PPAD-hard even with full information \cite{daskalakis2013complexity}. Notably, there have been attempts to establish asymptotic convergence with independent learning dynamics \cite{marden2007regret, marden2009payoff,young2004strategic} for weakly acyclic games \cite{young2020individual}, which includes potential games as a special case.

\subsection{Notation and paper organization}
We use $\Delta(\mathcal{A})$ to denote the probability simplex over the set $\mathcal{A}$. For a vector $\mathbf{a} \in \mathcal{A}^N$, we use $a_i \in \mathcal{A}$ and $a_{-i}\in\mathcal{A}^{N-1}$ to denote the entry with index $i$ and all the rest entries as a vector, respectively. The application of scalar functions such as $\exp$ and $\log$ to vectors are defined in an entry-wise fashion. Let $\mathbf{1}$ be the all-one vector, and $[N]=\{1,\ldots, N\}$. Given two distributions $\pi_1$ and $\pi_1'$ over $\mathcal{A}$, the Kullback-Leibler (KL) divergence from $\pi_1'$ to $\pi_1$ is defined by $\KL{\pi_1}{\pi_1'} = \sum_{a\in \mathcal{A}} \pi_1(a) (\log \pi_1(a) - \log \pi_1'(a))$. Note that KL divergence is additive for product distributions in the sense that $\KL{\pi}{\pi'}=\sum_{i\in[N]} \KL{\pi_i}{\pi_i'}$ for $\pi =  \pi_1 \times \cdots \times \pi_N  \in \Delta (\cA)^N$ and $\pi' =  \pi_1' \times \cdots \times \pi_N' \in \Delta (\cA)^N$. We denote Jeffrey divergence \cite{jeffreys1998theory} by
$J(\pi, \pi') = \KL{\pi}{\pi'} + \KL{\pi'}{\pi}$, which is the symmetric version of the KL divergence.

The rest of this paper is organized as follows. Section \ref{sec:formulation} presents the backgrounds of the potential game setup. Section \ref{sec:NPG_results} introduces independent NPG methods and presents the finite-time global convergence guarantees. Section \ref{sec:pf_sketch} provides an outline to the analysis and the rest of the proofs are deferred to the appendix. Section \ref{sec:numerical} presents numerical results to verify the theoretical findings. Finally, we conclude in Section \ref{sec:conclusion}.

%% file: potential_game.tex
\section{Potential Games with Entropy Regularization}
\label{sec:formulation} 
In this section, we introduce the basics of potential games, as well as the incorporation of entropy regularization into its formulation.
 
    \subsection{Potential games}
    A strategic game $\mathcal{G} = \{N, \cA, \{u_i\}_{i\in [N]}\}$ consists of $N$ agents each with an individual utility or payoff function 
    $$u_i: \cA^N \to [0, 1], \qquad i \in [N],$$ 
    where $\cA$ is, without loss of generality, a finite action space shared by all agents. The policy or mixed strategy
  of agent $i$ is denoted by  $\pi_i \in \Delta(\cA)$, which is a distribution over the action space $\cA$.
    By an abuse of notation, let $u_i(\pi)$ denote agent $i$'s expected utility function under the joint policy $\pi = \pi_1 \times \cdots \times \pi_N \in \Delta (\cA)^N$, i.e., 
    \begin{equation*}
        u_i(\pi) = \ex{a_i \sim \pi_i, \forall i \in [N]}{u_i(\bm{a})},
    \end{equation*}
    where we denote the action profile $(a_1, \cdots, a_N)$ of all agents by $\bm{a} \in \cA^N$. We shall often instead write $\bm{a} = (a_i, a_{-i})$ where $a_{-i} = \{a_j\}_{j\neq i}$ collects the actions of all agents but $i$; similarly, we write $\pi= (\pi_i, \pi_{-i})$, where $\pi_{-i} = \{\pi_j\}_{j\neq i}$ collects the policies of all agents but $i$.

  The game $\mathcal{G}$ is said to be a potential game if there exists a potential function $\Phi: \cA^N \to \mathbb{R}$ such that
    \[
        u_i(a_i,\, a_{-i}) - u_i(a_i',\, a_{-i}) = \Phi(a_i,\, a_{-i}) - \Phi(a_i',\, a_{-i})
    \]
    for any $a_i, a_i' \in \cA$, $a_{-i}\in \cA^{N-1}$ and $i \in[N]$.
    We assume that 
    \begin{equation}\label{eq:bound_Phi}
    0 \le \Phi(\bm{a}) \le \Phi_{\max}, \qquad \forall \bm{a} \in \cA^N,
    \end{equation} 
    where $\Phi_{\max}$ upper bounds the potential function. 
An important special case of the potential game is when all the agents share the same utility function, known as the identical-interest game \cite{monderer1996fictitious}. It is straightforward to see that for an identical-interest game, we can set $\Phi = u_i$ for all $i \in [N]$, and therefore $\Phi_{\max}=1$ due to the fact that the individual payoff is bounded in $[0,1]$.

    By linearity of expectation, we have
    \[
        u_i(\pi_i,\, \pi_{-i}) - u_i(\pi_i',\, \pi_{-i}) = \Phi(\pi_i,\, \pi_{-i}) - \Phi(\pi_i',\, \pi_{-i}),
    \]
    where, again with slight abuse of notation, we denote
    $$\Phi(\pi) = \ex{\bm{a}\sim \pi }{\Phi(\bm{a})} =  \ex{a_i\sim \pi_i, \forall i \in [N]}{\Phi(\bm{a})},$$ for any $\pi_i, \pi_i' \in \Delta(\cA)$, $\pi_{-i}\in \Delta(\cA)^{N-1}$ and $i \in [N]$.

\paragraph{Nash equilibrium.} We now introduce the important notion of \textit{Nash equilibrium} in a potential game.
\begin{definition}[Nash equilibrium]
    A joint policy $\pi^\star$ is called a \textit{Nash equilibrium} (NE) when it holds that
    \[
        u_i(\pi_i,\, \pi_{-i}) \ge u_i(\pi_i',\, \pi_{-i}),\quad \forall {\pi_i' \in \Delta(\cA)}, \; \forall i \in [N].
    \]
 \end{definition}
    In other words, every agent cannot improve its utility function by deviating from the current policy. It is known that there exists at least one NE in a strategic game with finite agents and actions \cite{nash1951non}. It follows immediately that the policy or strategy profile maximizing $\Phi$ in a potential game is an NE.

\paragraph{Marginalized utility.}	Before continuing, let us introduce an important quantity  called the marginalized utility    $r_i^{\pi}: \cA \to \mathbb{R}$:
    \begin{equation}\label{eq:marginalized_utility}
        r_i^{\pi}(a) = \ex{a_{-i} \sim \pi_{-i}}{u_i(a,\, a_{-i})},
    \end{equation}
    which can be viewed as the ``single-agent'' payoff or reward function when the policies of other agents are fixed. It is immediate to see that the utility function $u_i$ can be written as
    \begin{align*}
        u_i(\pi) &= \ex{\bm{a} \sim \pi}{u_i(\bm{a})}= \ex{a\sim\pi_i}{r_i^{\pi}(a)} = \langle r_i^{\pi}, \pi_i \rangle.
    \end{align*}
Here and throughout this paper, we shall often abuse the notation to treat $\pi$, $\pi_i$ and $r_i^{(t)}$ as vectors.
    
    \subsection{Entropy-regularized potential games}

    The \textit{quantal response equilibrium} (QRE) is proposed by McKelvey and Palfrey \cite{mckelvey1995quantal} as a seminal extension to the Nash equilibrium, which enables players to combat randomness in payoffs. A QRE or logit equilibrium $\pi_\tau^\star = \pi_{\tau, 1}^\star\times \cdots \times \pi_{\tau, N}^\star$ necessitates every agent to maximize its own utility function with entropy regularization \cite{mertikopoulos2016learning}, i.e.,
    \[
        u_{i,\tau}(\pi_{\tau,i}^\star,\, \pi_{\tau,-i}^\star) \ge u_{i,\tau}(\pi_i',\, \pi_{\tau,-i}^\star),\quad \forall{\pi_i' \in \Delta(\cA)},
    \]
    where the entropy-regularized individual utility function is given by 
    \[
        u_{i,\tau}(\pi) = u_i(\pi) + \tau \cH(\pi_i).
    \]
    Here, $\pi = \pi_1 \times \cdots \times \pi_N$, $\tau > 0$ is the regularization parameter, and $\cH(\pi_i) = -\sum_{a\in\cA} \pi_i(a|s) \log \pi_i(a|s)$ is the Shannon entropy of the policy  $\pi$ employed by agent $i$.     By introducing the regularized potential function
    \[
        \Phi_{\tau}(\pi) =  \Phi(\pi) + \tau \cH(\pi) := \Phi(\pi) + \tau \sum_{i\in [N]} \cH(\pi_i),
    \]
    it is easy to verify
    \[
        u_{i,\tau}(\pi_i,\, \pi_{-i}) - u_{i,\tau}(\pi_i',\, \pi_{-i}) = \Phi_{\tau}(\pi_i,\, \pi_{-i}) - \Phi_{\tau}(\pi_i', \,\pi_{-i}).
    \]
    for any $\pi_i, \pi_i' \in \Delta(\cA), \pi_{-i}\in \Delta(\cA)^{N-1}$ and $i \in [N]$, as long as the unregularized game is a potential game.
    
\paragraph{Fixed-point characterization of QRE.}  An equivalent interpretation of QRE is to let each agent assign the probability mass in its policy according to every action's utility in a bounded rationality fashion \cite{selten1989evolution}:
    \begin{equation}\label{eq:QRE}
        \pi_{\tau, i}^\star(a) \propto \exp\left(r_i^{\pi_\tau^\star}(a)/\tau \right), \quad \forall i \in [N],
    \end{equation}
where $r_i^{\pi_\tau^\star}$ is the marginalized utility of $\pi_\tau^\star$ defined in \eqref{eq:marginalized_utility}. Note that the above relation defines a fixed-point equation of $\pi_\tau^\star$.

%% file: entropy_potential_game.tex
\section{Finite-Time Global Convergence of Independent Natural Policy Gradient Methods}

\label{sec:NPG_results}

A popular approach in the game theory literature to find an NE of a potential game is for each agent to switch to the best or better response policy, one at a time, and is generally referred to as \textit{best-response dynamics}. This approach converges to an NE in finite iterations \cite{monderer1996potential} and underlies the algorithm design of a considerable number of works on, e.g., cut games \cite{christodoulou2006convergence}, congestion games \cite{chien2011convergence}, weakly acyclic games \cite{young2004strategic}, and, more recently, their extensions in the Markovian setting \cite{song2021can,arslan2016decentralized}. It is noted, however, that this approach isolates itself from the independent learning paradigm as the update sequence needs to be scheduled in a centralized manner that is not often possible.
Therefore, it is greatly desirable to design independent update rules, where each agent updates simultaneously without observing the payoffs of other agents, that achieves faster convergence. In this section, we answer this call by developing the independent natural policy gradient method to solve (entropy-regularized) potential games with finite-time global convergence guarantees.

\subsection{Independent natural policy gradient method}

In policy optimization, it is common practice to parameterize the policy class in a way that obviates the need for tackling probability simplex constraint. We consider the standard softmax parameterization, where every agent $i$ generates its own policy $\pi_{\theta_i}$ parameterized with $\theta_i \in \mathbb{R}^{|\cA|}$ through the softmax transform:
\[
    \pi_{\theta_i} (a) = \frac{\exp(\theta_i(a))}{\sum_{a\in\cA}\exp(\theta_i(a))}.
\]

Every agent $i$ evaluates and updates its policy independently using the \textit{natural policy gradient} (NPG) method \cite{kakade2001natural}:
\begin{equation}        \label{eq:def_NPG}
    \theta_i \leftarrow \theta_i + \eta (\mathcal{F}^{\theta_i})^\dagger \nabla_{\theta_i} u_{i,\tau}(\pi),
\end{equation}
where $(\mathcal{F}^{\theta_i})^\dagger$ denotes the Moore-Penrose pseudo-inverse of the Fisher information matrix $\mathcal{F}^{\theta_i}$, which is defined as 
\[
    \mathcal{F}^{\theta_i} = \ex{a \sim \pi_{\theta_i}(\cdot)}{(\nabla_{\theta_i} \log\pi_{\theta_i}(a))(\nabla_{\theta_i} \log\pi_{\theta_i}(a))^\top},
\]
and $\eta > 0$ is the learning rate. Moreover, the gradient $\nabla_{\theta_i} u_{i,\tau}(\pi)$ can be expressed as
\[
    \nabla_{\theta_i} u_{i,\tau}(\pi) = r_i^{\pi} - \tau \log\pi_i - \tau \mathbf{1}.
\]
It turns out that with some algebra, the NPG update rule \eqref{eq:def_NPG} can be equivalently rewritten with respect to the policies in use \cite{cen2020fast}:
\begin{equation}
    \label{eq:NPG_update}
    \pi_i^{(t+1)}(a)\propto \pi_i^{(t)}(a)^{1-\eta\tau} \exp(\eta r^{(t)}_i(a)),
\end{equation}
where $\pi_i^{(t)}$ denotes agent $i$'s policy in the $t$-th iteration, and $r_i^{(t)}: = r_i^{\pi^{(t)}}$ denotes the marginalized utility of $\pi^{(t)}$ (cf.~\eqref{eq:marginalized_utility}). The complete procedure is summarized in Algorithm~\ref{alg:INPG}.

\begin{algorithm}[ht]
\label{alg:INPG}
\caption{Independent NPG for Entropy-regularized Potential Games}
\textbf{Input:} Regularization parameter $\tau > 0$, step size for policy update $\eta > 0$.\\
\textbf{Initialization:} Set $\pi_i^{(0)}$ as uniform policy for all $i \in [N]$.
\SetKwProg{ForP}{for}{ do in parallel}{end}

\For{$t = 0,1,\cdots$}{
\ForP{all agent $i\in [N]$}{
Observe agent $i$'s marginalized utility $r_i^{(t)}$.

Perform policy update
\begin{equation*}
\pi_i^{(t+1)}(a)\propto \pi_i^{(t)}(a)^{1-\eta\tau} \exp(\eta r^{(t)}_i(a)).
\end{equation*}
}
}

\end{algorithm}

To better understand the update rule \eqref{eq:NPG_update} as well as prepare for follow-up analysis, we introduce $\pi_i^{\star(t)}$ to denote agent $i$'s best-response policy in the $t$-th iteration, which is the policy that obeys
\begin{equation}\label{eq:def_br}
    u_{i, \tau}(\pi_i^{\star(t)},\, \pi_{-i}^{(t)}) = \max_{\pi_i'}u_{i, \tau}(\pi_i',\, \pi_{-i}^{(t)}).
\end{equation}
It is easily seen that
\begin{equation}\label{eq:best_response}
    \pi_i^{\star(t)}(a) \propto \exp(r_{i}^{(t)}(a)/\tau).
\end{equation}
Therefore, the updated policy in \eqref{eq:NPG_update} can be regarded as a multiplicative combination of the current policy $\pi_i^{(t)}$ and the best-response policy $\pi_i^{\star(t)}$, where the weight is controlled by the learning rate $\eta$. 
Note that the unregularized counterpart of the method is equivalent to Multiplicative Weights Update method (MWU) \cite{littlestone1994weighted,arora2012multiplicative} or Hedge \cite{freund1999adaptive}.

\subsection{Finite-time global convergence}  
We are now ready to present our main theorem concerning the finite-time global convergence of independent NPG for solving entropy-regularized potential games. We introduce
    \begin{align*}
        \texttt{NE-gap}(\pi) &= \max_{i \in [N], \pi_i'\in \Delta(\cA)} \brk{u_i(\pi_i', \pi_{-i}) - u_i(\pi_i, \pi_{-i})}
    \end{align*}
and
\begin{align*}  
        \texttt{QRE-gap}_\tau(\pi) &= \max_{i \in [N], \pi_i'\in \Delta(\cA)} \brk{u_{i,\tau}(\pi_i', \pi_{-i}) - u_{i,\tau}(\pi_i, \pi_{-i})}
    \end{align*} 
to characterize how close the joint policy $\pi$ is to an equilibrium.
A joint policy $\pi$ is said to be an $\epsilon$-QRE (resp. $\epsilon$-NE) when $\texttt{QRE-gap}(\pi) \le \epsilon$ (resp. $\texttt{NE-gap}(\pi) \le \epsilon$). For notational simplicity, we denote 
$$\Phi_\tau^{(t)}: = \Phi_\tau(\pi^{(t)}), \qquad \texttt{QRE-gap}_{\tau}^{(t)}:=\texttt{QRE-gap}_{\tau}(\pi^{(t)}),\quad \mbox{and} \quad \texttt{NE-gap}^{(t)}:=\texttt{NE-gap}(\pi^{(t)}).$$ 

Our main theorem is as follows, whose proof is deferred to Section~\ref{sec:pf_sketch}.
\begin{theorem}    \label{thm:QRE_convergence}
    Suppose that the learning rate $\eta$ satisfies $\eta \le \frac{1}{2(\min\{\sqrt{N}, 2\Phi_{\rm max}\}+\tau)}$, then for independent NPG updates \eqref{eq:NPG_update}, it holds that
    \begin{equation*}
        \frac{1}{T}\sum_{t=1}^{T} {\texttt{QRE-gap}_{\tau}^{(t)}} \le \frac{2}{\eta\tau T} \Big(\tau\norm{\log \pi^{(0)} - \log \pi^{\star(0)}}_\infty + \sqrt{2 \eta T(\Phi_\tau^{(T)} - \Phi_\tau^{(0)})}\Big).
    \end{equation*}
\end{theorem}
\smallskip
Theorem \ref{thm:QRE_convergence} suggests that the average iterate of independent NPG converges to an $\epsilon$-QRE at a sublinear rate when we initialize it via uniform policies, as indicated in the following corollary. The proof can be found in Appendix~\ref{proof:rate_QRE}.

\begin{corollary} \label{corollary:rate_QRE}
    Assume the independent NPG method is initialized with uniform policies at all agents. Setting the learning rate $\eta = 1/(2(\min\{\sqrt{N}, 2\Phi_{\rm max}\}+\tau))$ and $\tau = \mathcal{O}(1)$, then independent NPG updates ensure that 
    $
        \frac{1}{T}\sum_{t=1}^{T} \texttt{QRE-gap}_{\tau}^{(t)} \le \epsilon
    $
    with at most
    \[
        T = \mathcal{\mathcal{O}}\prn{\frac{\min\{\sqrt{N}, \Phi_{\rm max}\}\Phi_{\max}}{\tau^2\epsilon^2}}
    \]
    iterations. 
\end{corollary}

\paragraph{Finding approximate NEs.} It is possible to leverage the entropy-regularized potential game to find an approximate NE by setting the regularization parameter sufficiently small. Note that
\begin{align*} 
        \texttt{NE-gap}(\pi) &= \max_{i \in [N], \pi_i'\in \Delta(\cA)} \brk{u_i(\pi_i', \pi_{-i}) - u_i(\pi_i, \pi_{-i})}\\
        &\le \max_{i \in [N], \pi_i'\in \Delta(\cA)} \brk{u_{i,\tau}(\pi_i', \pi_{-i}) - u_{i,\tau}(\pi_i, \pi_{-i})} + \max_{i \in [N], \pi_i'\in \Delta(\cA)}\brk{-\tau \cH(\pi_i') + \tau \cH(\pi_i)}\\
        &\le \texttt{QRE-gap}_\tau(\pi) + \tau \log |\cA|.
\end{align*}
Therefore, by setting the entropy regularization at
$$ \tau = \frac{\epsilon}{2\log|\mathcal{A}|} ,$$
with at most
\[
    T = \widetilde{\mathcal{O}}\prn{\frac{\min\{\sqrt{N}, \Phi_{\max}\} \Phi_{\max}}{\epsilon^4}}
\]
iterations, we can ensure 
$\frac{1}{T}\sum_{t=1}^{T} \texttt{NE-gap}^{(t)} \le \epsilon
$.
 
 \paragraph{Comparisons with prior art.}
Importantly, our iteration complexities do not depend on the size of the action space (up to logarithmic factors), which is in sharp contrast to existing analyses of potential games using other policy gradient approaches, such as direct PG \cite{zhang2021gradient, leonardos2021global, ding2022independent, mao2022improving} and NPG with log-barrier regularization \cite{zhang2022effect}, where the iteration complexity scales as  $\widetilde{\mathcal{O}}\prn{{N |\cA| \Phi_{\max}}/{\epsilon^2}}$ to find an $\epsilon$-approximate NE. In comparison, while our rate $\widetilde{\mathcal{O}}\prn{{\min\{\sqrt{N}, \Phi_{\max}\} \Phi_{\max}}/{\epsilon^4}}$ is worse in terms of $\epsilon$, it is almost independent of the size $|\cA|$ of the action space, as well as exhibits only a sublinear dependency with the number of agents $N$, thus can be beneficial for problems with large action spaces and a large number of agents. Furthermore, for the special case of identical-interest games \cite{monderer1996fictitious} where $\Phi_{\max} = 1$, the convergence rate of our method simplifies to
$$  \widetilde{\mathcal{O}}\prn{\frac{1}{\epsilon^4}} ,$$
 which leads to the first method that achieves a dimension-free iteration complexity (up to a logarithmic factor) for finding an $\epsilon$-NE without imposing any isolation assumptions.


%% file: pf-matrix.tex
\section{Proof of Theorem~\ref{thm:QRE_convergence}}
\label{sec:pf_sketch}

Before proceeding to the main proof, we first record two useful lemmas. The following elementary lemma is standard (see e.g., \cite[Lemma 3]{cen2021fast} and \cite{cen2020fast}) and will be helpful in the analysis. 
\begin{lemma}
    \label{lem:log_pi_gap}
    For any $\mu_1, \mu_2 \in \Delta(\mathcal{A})$ satisfying 
    $$\mu_1(a) \propto \exp(x_1(a))\quad \mbox{and} \quad
        \mu_2(a) \propto \exp(x_2(a))    $$
    for some $x_1, x_2 \in \mathbb{R}^{|\mathcal{A}|}$,
    we have
    \begin{equation}
        \norm{\log \mu_1 - \log \mu_2}_\infty \le 2 \norm{x_1 - x_2}_\infty.
    \end{equation}
\end{lemma}
\smallskip

Another useful lemma connects the marginalized utility with the policy, as given below.
\begin{lemma}
    \label{lem:r_lip}
    Given any $\pi, \pi' \in \Delta(\cA)^N$, the difference in the marginalized utility (cf.~\eqref{eq:marginalized_utility}) can be bounded by
    \[
    \norm{r_i^{\pi} - r_i^{\pi'}}_\infty \le \sqrt{J(\pi, \pi')},
    \]
    where $J(\pi, \pi')= \KL{\pi}{\pi'} + \KL{\pi'}{\pi}$ is the Jeffrey divergence.
\end{lemma}
\smallskip
\begin{proof}
See Appendix~\ref{sec:pf:lem:r_lip}.
\end{proof}

\subsection{Step 1: quantify the policy improvement} 

We start by the following key lemma that gives a lower bound of the improvement in terms of the regularized potential function $\Phi_\tau^{(t)}$.

\begin{lemma}
    \label{lem:bandit_perf_impv}
  The independent NPG update \eqref{eq:NPG_update} guarantees that
    \begin{align*}
        &\Phi_\tau^{(t+1)} - \Phi_\tau^{(t)}\ge \prn{\frac{1}{\eta}-\min\{\sqrt{N}, 2\Phi_{\max}\}-\tau} J\prn{\pi^{(t+1)},\pi^{(t)}}.
    \end{align*}
\end{lemma}
\smallskip
\begin{proof}
See Appendix~\ref{sec:pf:lem:bandit_perf_impv}.
\end{proof}

Lemma \ref{lem:bandit_perf_impv} ensures the monotonic improvement of the regularized potential function $\Phi_{\tau}$ when $\eta$ is not too large. Specifically, setting $\eta \le 1/(2(\min\{\sqrt{N}, 2\Phi_{\max}\}+\tau))$, we have
\begin{align*}
    &\Phi_\tau^{(t+1)} - \Phi_\tau^{(t)} \ge \frac{1}{2\eta} J\prn{\pi^{(t+1)},\pi^{(t)}},
\end{align*}
which is guaranteed to be non-negative. Summing the above inequality over $t = 0, \cdots, T-1$ gives
\begin{equation}
\begin{aligned}
    &\sum_{t=0}^{T-1} J\prn{\pi^{(t+1)},\pi^{(t)}} \le 2\eta(\Phi_\tau^{(T)} - \Phi_\tau^{(0)}),
    \label{eq:KL_sum_bound}        
\end{aligned}    
\end{equation}
which controls the change of $\pi^{(t)}$ over time $t$ via the size of the regularized potential function.

\subsection{Step 2: introduce the auxiliary sequence}

Motivated by \cite{cen2021fast,cen2020fast}, we introduce an auxiliary sequence $\{\xi_i^{(t)}\in \mathbb{R}^{|\mathcal{A}|}, \; i \in [N]\}$, constructed recursively by
\begin{subequations}
    \begin{align}
        \xi_i^{(0)}(a) &= \norm{\exp(r_i^{(0)}/\tau)}_1 \cdot \pi_i^{(0)}(a),\\
        \xi_i^{(t+1)}(a) &= {\xi_i^{(t)}}(a)^{1-\eta\tau}\exp(\eta r_i^{(t)}(a)). \label{eq:xi_def}
    \end{align}
\end{subequations}
Compared with the independent NPG update rule \eqref{eq:NPG_update}, it is clear that $\xi_i^{(t)} \propto \pi_i^{(t)}$ up to normalization. In addition, we have
\begin{align*}
     \log\xi_i^{(t+1)} - r_i^{(t+1)}/\tau  
    &= (1-\eta\tau) \log \xi_i^{(t)} + \eta r_i^{(t)} -  r_i^{(t+1)}/\tau\\
    &= (1-\eta\tau) \prn{\log \xi_i^{(t)} - r_i^{(t)}/\tau} + \prn{r_i^{(t)} -  r_i^{(t+1)}}/\tau,
\end{align*}
which implies 
\begin{align}
     \norm{\log\xi_i^{(t+1)} - r_i^{(t+1)}/\tau}_\infty  
    &\le (1-\eta\tau) \norm{\log \xi_i^{(t)} - r_i^{(t)}/\tau}_\infty + \norm{r_i^{(t)} -  r_i^{(t+1)}}_\infty/\tau \nonumber \\
    &\le (1-\eta\tau)^{t+1}\norm{\log \xi_i^{(0)} - r_i^{(0)}/\tau}_\infty  + \tau^{-1} \sum_{s=0}^{t}(1-\eta\tau)^{t-s}\norm{r_i^{(s)} -  r_i^{(s+1)}}_\infty \nonumber\\
    &\le (1-\eta\tau)^{t+1}\norm{\log \xi_i^{(0)} - r_i^{(0)}/\tau}_\infty  + \tau^{-1} \sum_{s=0}^{t}(1-\eta\tau)^{t-s} \sqrt{J\prn{\pi^{(s+1)},\pi^{(s)}}},\label{eq:recursion}
\end{align}
where the last line follows by applying Lemma \ref{lem:r_lip} to the last term by setting $\pi = \pi^{(t)}$ and $\pi' = \pi^{(t+1)}$.

\subsection{Step 3: bound the gap}

Note that by the definition of the best-response policy in \eqref{eq:def_br}, the term of interest in $ \texttt{QRE-gap}_\tau^{(t)} $ can be controlled as
\begin{align*}
  u_{i,\tau}(\pi_i^{\star(t+1)},\, \pi_{-i}^{(t+1)}) - u_{i,\tau}(\pi_i^{(t+1)},\, \pi_{-i}^{(t+1)})   
    &= \innprod{\pi_i^{\star(t+1)} - \pi_i^{(t+1)},\, r_i^{(t+1)}} + \tau \cH(\pi_i^{\star(t+1)}) - \tau\cH(\pi_i^{(t+1)})\\
    &= \tau \KL{\pi_i^{(t+1)}}{\pi_i^{\star(t+1)}} \le \tau\norm{\log\pi_i^{(t+1)} - \log \pi_i^{\star(t+1)}}_\infty \\
    &  \le 2\tau \norm{\log\xi_i^{(t+1)} - r_i^{(t+1)}/\tau}_\infty,
\end{align*}
where the first line follows from the definition \eqref{eq:marginalized_utility}, the second step results from a direct consequence of \eqref{eq:best_response}:
\begin{align*}
 \innprod{\pi_i^{\star(t+1)} - \pi_i^{(t+1)},\, r_i^{(t+1)}}  & =  \innprod{\pi_i^{\star(t+1)} - \pi_i^{(t+1)},\, \tau \log \pi_i^{\star(t+1)} } 
\end{align*}
with a little algebra, and the last line follows from Lemma \ref{lem:log_pi_gap}. 
Taking maximum over $i \in [N]$, in conjunction with \eqref{eq:recursion}, we end up with
\begin{align*}
    {\texttt{QRE-gap}_{\tau}^{(t+1)}}&\le 2\tau(1-\eta\tau)^{t+1}\norm{\log \pi^{(0)} - \log\pi^{\star(0)}}_\infty + 2\sum_{s=0}^{t} (1-\eta\tau)^{t-s} \sqrt{J\prn{\pi^{(s+1)},\pi^{(s)}}}.
\end{align*}
%

Summing the inequality over $t=0,\ldots, T-1$ gives
\begin{align*}
    &\sum_{t=0}^{T-1} \texttt{QRE-gap}_{\tau}^{(t+1)}\\
    &\le 2\tau\sum_{t=0}^{T-1} (1-\eta\tau)^{t+1} \max_{i\in[N]} \norm{\log \pi_i^{(0)} - \log \pi_i^{\star(0)}}_\infty + 2 \sum_{t=0}^{T-1}\sum_{s=0}^{t}(1-\eta\tau)^{t-s} \sqrt{J\prn{\pi^{(s+1)},\pi^{(s)}}}\\
    &\le \frac{2}{\eta\tau}\Big(\tau\norm{\log \pi^{(0)} - \log \pi^{\star(0)}}_\infty + \sum_{s=0}^{T-1}\sqrt{J\prn{\pi^{(s+1)},\pi^{(s)}}}\Big).
\end{align*}
The proof is thus completed by noticing
\begin{align*}
    \sum_{s=0}^{T-1}\sqrt{J\prn{\pi^{(s+1)},\pi^{(s)}}} &\le \sqrt{T\sum_{s=0}^{T-1} J\prn{\pi^{(s+1)},\pi^{(s)}}} \le \sqrt{2\eta T(\Phi_\tau^{(T)} - \Phi_\tau^{(0)})}.
\end{align*}
Here, the second step results from Pinsker's inequality, and the last line follows from  \eqref{eq:KL_sum_bound}.


%% file: numerical.tex
\section{Numerical Experiments}
\label{sec:numerical}

We examine the performance of independent NPG methods --- in comparison with policy gradient (PG) methods with direct parametrization --- on a potential game with $N = 4$ agents and an action space with $|\cA| = 20$ actions. The potential function $\Phi(\bm{a})$ is independently drawn from a Beta distribution $\text{Beta}(\frac{1}{2},\frac{1}{2})$ for each $\bm{a} \in \cA^N$. We set the learning rate as $\eta = \frac{1}{2(\sqrt{N}+\tau)}$ for the independent NPG methods, while PG with direct parametrization adopts $\eta = \frac{1}{2N|\cA|}$, which is the maximum possible learning rate prescribed in \cite{zhang2021gradient, leonardos2021global, ding2022independent}. Figure \ref{fig:potential} verifies the monotonic improvement of regularized potential function $\Phi_\tau^{(t)}$ for the NPG method, corroborating our theoretical analysis.

Figure \ref{fig:NE} shows the averaged performance of various methods over 10 independent runs, in terms of finding approximate NE for the unregularized potential games, and Figure \ref{fig:QRE} plots that of finding QRE. While independent NPG with a larger regularization parameter $\tau = 10^{-2}$ converges to QRE faster, its $\texttt{NE-gap}$ stalls due to the presence of large entropy regularization. 
With $\tau = 10^{-3}$, NPG achieves a better trade-off and finds a much smaller $\texttt{NE-gap}$. %

\begin{figure}[h]
    \centering
    \includegraphics[width=0.5\linewidth]{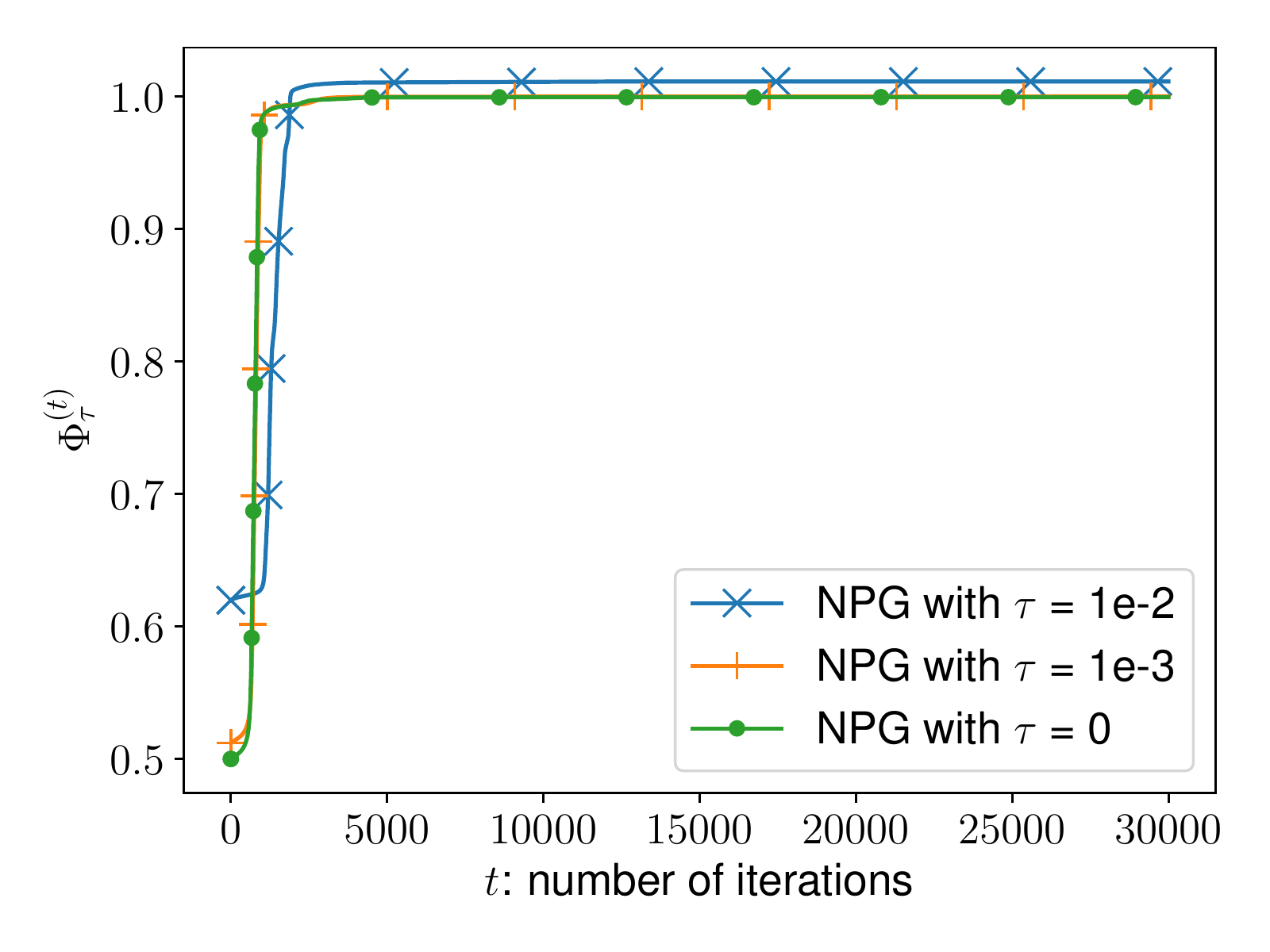}
    \caption{The regularized potential function $\Phi_\tau^{(t)}$ versus the iteration count of NPG with various entropy regularization parameter $\tau$ and softmax parameterization.}
    \label{fig:potential}
\end{figure}

\begin{figure}[h]
    \centering
    \includegraphics[width=0.5\linewidth]{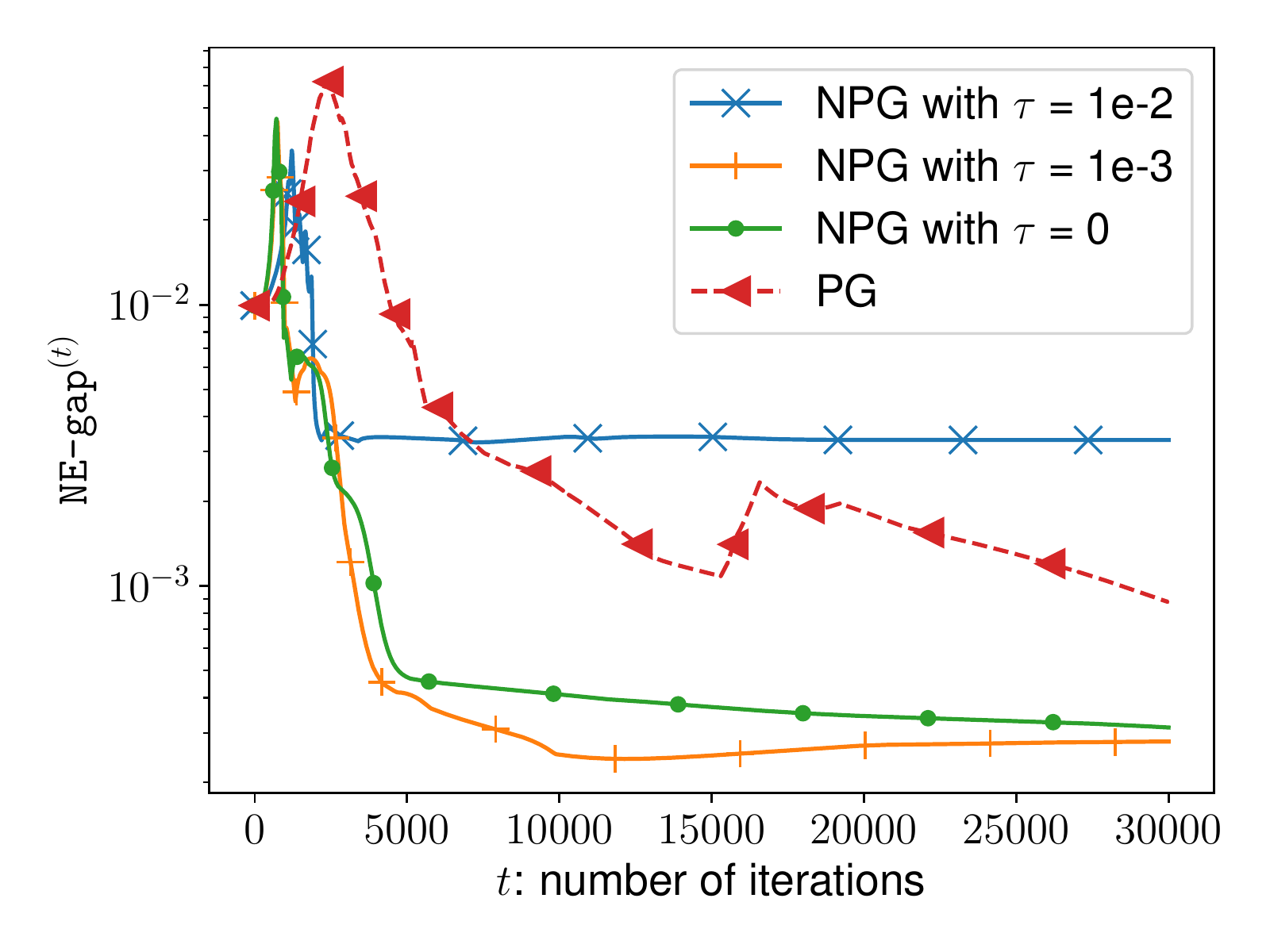}
    \caption{$\texttt{NE-gap}^{(t)}$ versus the iteration count of unregularized PG with direct parameterization and NPG with various entropy regularization parameter $\tau$ and softmax parameterization.}
    \label{fig:NE} 
\end{figure}

\begin{figure}[h]
    \centering
    \includegraphics[width=0.5\linewidth]{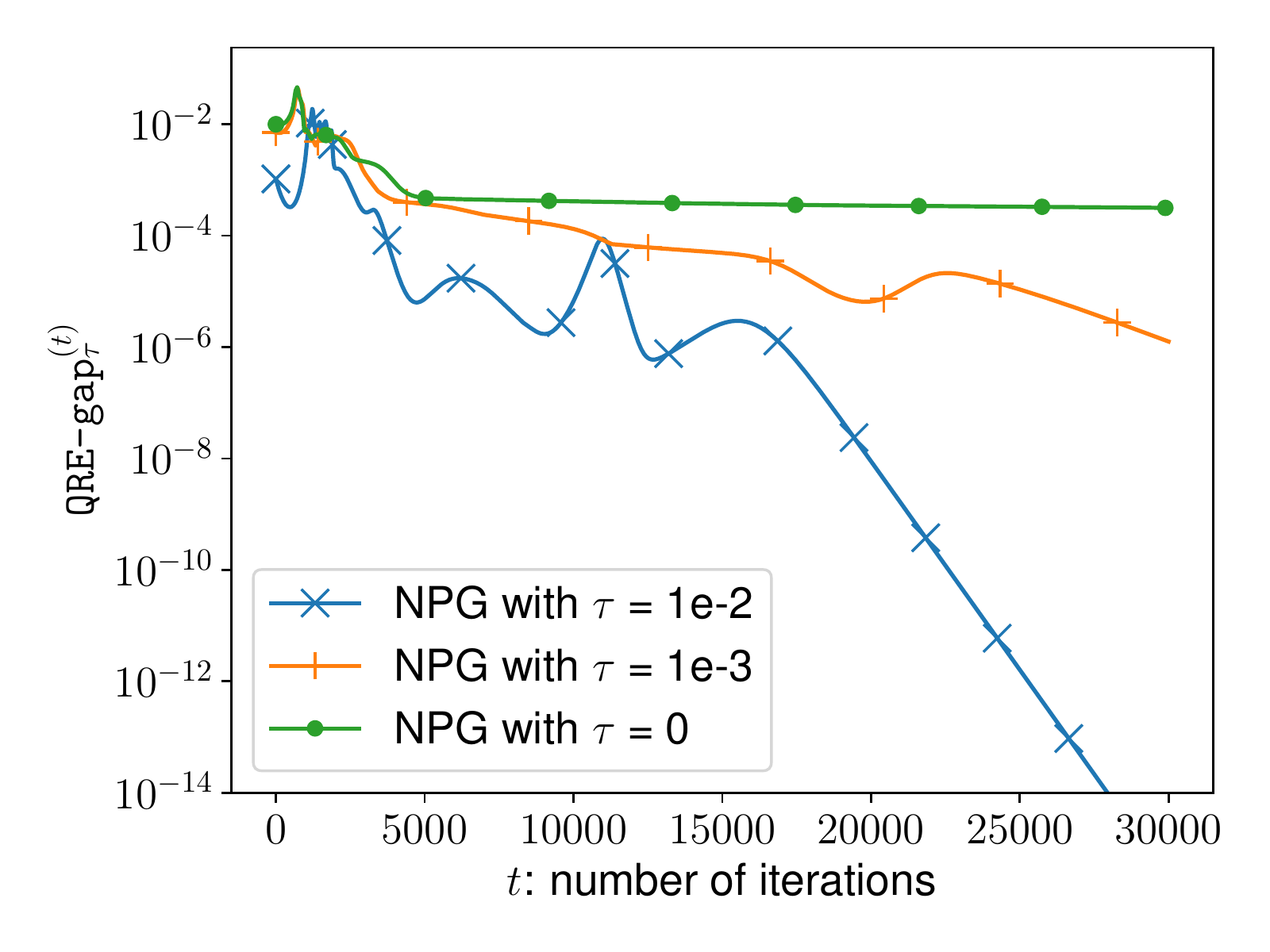}
    \caption{$\texttt{QRE-gap}_\tau^{(t)}$ versus the iteration count of NPG with various entropy regularization parameter $\tau$ and softmax parameterization.}
    \label{fig:QRE}
\end{figure}

%% file: conclusions.tex
\section{Conclusions and Discussions}
\label{sec:conclusion}
This paper studies independent NPG methods for entropy-regularized potential games and develops a sublinear rate of convergence to quantum response equilibrium, which is independent of the size of the action spaces up to logarithmic factors and grows only sublinearly with respect to the number of agents. In addition, the method achieves the first {\em dimension-free} convergence rate for the important special case of identical-interest games, where the rate is independent of both the size of the action space and the number of agents. The approach can also be used as a smoothing technique to find Nash equilibria by setting the regularization parameter sufficiently small, without imposing the isolation assumption as often required in prior works. This work leaves open a number of interesting questions: 
\begin{itemize}
\item Can we tighten the convergence rate in terms of the dependencies on $\epsilon$? 
\item Can we extend the analysis to establish finite-time global convergence for Markov potential games? 
\end{itemize}
We leave the answers to future work.

%% file: pf-lemma.tex
\section{Proof of Lemma \ref{lem:r_lip}}
\label{sec:pf:lem:r_lip}
    Given any $\pi, \pi' \in \Delta(\cA)^N$, we have
    \begin{equation}
    \begin{aligned}
         \left|r_i^{\pi}(a) - r_i^{\pi'}(a)\right|        &=\left|\ex{a_{-i}\sim {\pi}_{-i}}{u_i(a, \,a_{-i})} - \ex{a_{-i}\sim {\pi}_{-i}'}{u_i(a, \,a_{-i})}\right|\\
        &\overset{\mathrm{(i)}}{\le} 2\norm{u_i}_\infty d_{TV}\prn{{\pi}_{-i},\, {\pi}'_{-i}}\\
        &\le \sqrt{2 d_{TV}\prn{{\pi}_{-i},\, {\pi}'_{-i}}^2 + 2d_{TV}\prn{{\pi}'_{-i},\, {\pi}_{-i}}^2}\\
        &\overset{\mathrm{(ii)}}{\le} \sqrt{\KL{\pi_{-i}}{\pi'_{-i}} + \KL{\pi'_{-i}}{\pi_{-i}}} \\
        &\le \sqrt{\KL{\pi}{\pi'} + \KL{\pi'}{\pi}} =  \sqrt{J(\pi, \pi')},
        \label{eq:r_tilde_diff}    
    \end{aligned}
    \end{equation}
    where $d_{TV}\prn{\cdot,\,\cdot}$ refers to total variation distance. Here, $(\mathrm{i})$ follows from applying $\left|\int_\Omega h {\rm d}\mu - \int_\Omega h {\rm d}\nu \right| \le 2 d_{TV}(\mu, \nu) \norm{h}_\infty$ which holds for any probability measures $\mu$, $\nu$ and bounded measurable function $h: \Omega \to \mathbb{R}$ (see e.g., \cite[Corollary 13.4]{Driver2007notes}), and $(\mathrm{ii})$ results from Pinsker's inequality.  

\section{Proof of Lemma \ref{lem:bandit_perf_impv}}
\label{sec:pf:lem:bandit_perf_impv}

The proof is composed of two parts, each establishing the following bounds 
\begin{subequations} \label{eq:ensemble}
\begin{align}
\Phi_\tau^{(t+1)} - \Phi_\tau^{(t)} & \ge \prn{\frac{1}{\eta}-  \sqrt{N} -\tau} J\prn{\pi^{(t+1)},\pi^{(t)}}, \label{eq:perf_impv_part1} \\
\Phi_\tau^{(t+1)} - \Phi_\tau^{(t)}& \ge \prn{\frac{1}{\eta}-  2\Phi_{\max}-\tau} J\prn{\pi^{(t+1)},\pi^{(t)}} \label{eq:perf_impv_part2} 
\end{align}
\end{subequations}
respectively. Combining the two bounds then finishes the proof.

\subsection{Proof of \eqref{eq:perf_impv_part1} }
We introduce
\begin{equation*}
    \widetilde{\pi}_{-i}^{(t)}(a_{-i}) = \prod_{j<i}\pi_j^{(t)}(a_j) \prod_{k>i}\pi_k^{(t+1)}(a_k) \in \Delta(\cA)^{N-1}
\end{equation*}
to denote the mixed strategy profile (except that of agent $i$) where the agents with index $j<i$ follow $\pi_j^{(t)}$ and the agents with index $k>i$ follow $\pi_k^{(t+1)}$ instead. Let $\widetilde{r}_i^{(t)}$ be the associated marginalized utility function, i.e.,
\begin{align} 
        \widetilde{r}_i^{(t)}(a) &= \exlim{a_i = a, a_{-i}\sim \widetilde{\pi}_{-i}^{(t)}}{u_i(\bm{a})} \label{eq:def_tilde_br}\\
        &= \sum_{a_{-i} \in \mathcal{A}^{N-1}}  u_i(a,\, a_{-i})\prod_{j<i}\pi_j^{(t)}(a_j) \prod_{k>i}\pi_k^{(t+1)}(a_k). \nonumber
\end{align}
It follows from the above definition that we have
\begin{align}
    \Phi_\tau (\pi_{i}^{(t)}, \widetilde{\pi}_{-i}^{(t)}) &= \Phi_\tau (\pi_1^{(t)},\cdots,\pi_{i}^{(t)},\pi_{i+1}^{(t+1)},\cdots,\pi_N^{(t+1)})  = \Phi_\tau (\pi_{i+1}^{(t+1)}, \widetilde{\pi}_{-(i+1)}^{(t)}) \label{eq:phi_transfer}
\end{align}
for $i \in [N-1]$.

We now decompose $\Phi_\tau^{(t+1)} - \Phi_\tau^{(t)}$ as follows:
\begin{align*}
    \Phi_\tau^{(t+1)} - \Phi_\tau^{(t)}
    &= \Phi_\tau(\pi_{1}^{(t+1)}, \widetilde{\pi}_{-1}^{(t)})- \Phi_\tau(\pi_N^{(t)}, \widetilde{\pi}_{-N}^{(t)})\\
    &= \Phi_\tau(\pi_{1}^{(t+1)}, \widetilde{\pi}_{-1}^{(t)}) - \Phi_\tau(\pi_{1}^{(t)}, \widetilde{\pi}_{-1}^{(t)}) + \Phi_\tau(\pi_{1}^{(t)}, \widetilde{\pi}_{-1}^{(t)}) -\Phi_\tau(\pi_N^{(t)}, \widetilde{\pi}_{-N}^{(t)})\\
    &\overset{\mathrm{(i)}}{=} \brk{\Phi_\tau(\pi_{1}^{(t+1)}, \widetilde{\pi}_{-1}^{(t)}) - \Phi_\tau(\pi_{1}^{(t)}, \widetilde{\pi}_{-1}^{(t)})}  + \Phi_\tau(\pi_{2}^{(t+1)}, \widetilde{\pi}_{-2}^{(t)}) -\Phi_\tau(\pi_N^{(t)}, \widetilde{\pi}_{-N}^{(t)})\\
    &\overset{\mathrm{(ii)}}{=}  \sum_{i=1}^N\brk{\Phi_\tau(\pi_i^{(t+1)},\widetilde{\pi}_{-i}^{(t)}) - \Phi_\tau(\pi_i^{(t)}, \widetilde{\pi}_{-i}^{(t)}) }\\
    &= \sum_{i=1}^N\brk{u_{i,\tau}(\pi_i^{(t+1)},\widetilde{\pi}_{-i}^{(t)}) - u_{i,\tau}(\pi_i^{(t)}, \widetilde{\pi}_{-i}^{(t)}) }\\
    &= \sum_{i=1}^N\brk{\innprod{\widetilde{r}_i^{(t)}, \pi_i^{(t+1)} - \pi_i^{(t)}} + \tau \prn{\mathcal{H}(\pi_i^{(t+1)}) - \mathcal{H}(\pi_i^{(t)})}},
\end{align*}
where (i) follows from \eqref{eq:phi_transfer}, and (ii) follows from repeating the above process over all agents, and the last line follows from \eqref{eq:def_tilde_br}. For every $i \in [N]$, we have
\begin{align*}
    &\innprod{\widetilde{r}_i^{(t)}, \pi_i^{(t+1)} - \pi_i^{(t)}}+ \tau \prn{\mathcal{H}(\pi_i^{(t+1)}) - \mathcal{H}(\pi_i^{(t)})}\\
    &= {\innprod{r_i^{(t)}, \pi_i^{(t+1)} - \pi_i^{(t)}}+ \tau \prn{\mathcal{H}(\pi_i^{(t+1)}) - \mathcal{H}(\pi_i^{(t)})}} + \innprod{\widetilde{r}_i^{(t)} - r_i^{(t)}, \pi_i^{(t+1)} - \pi_i^{(t)}}.
\end{align*}
We control the terms separately.
\begin{itemize}
\item For the first two terms, recall that taking logarithm on both sides of \eqref{eq:NPG_update} gives
\begin{equation*}
    \eta r^{(t)}_i = \log \pi_i^{(t+1)} - (1-\eta\tau) \log \pi_i^{(t)} + c \mathbf{1}
\end{equation*}
for some constant $c$. It follows that
\begin{align}
     \langle r^{(t)}_i, \pi^{(t+1)}_i - \pi^{(t)}_i\rangle + \tau \prn{\mathcal{H}(\pi_i^{(t+1)}) - \mathcal{H}(\pi_i^{(t)})} 
    &= \frac{1}{\eta}\langle \log \pi^{(t+1)}_i - \log\pi^{(t)}_i, \pi^{(t+1)}_i - \pi^{(t)}_i\rangle\notag\\
    &\quad+ \tau \prn{\innprod{\log \pi_i^{(t)}, \pi_i^{(t+1)} - \pi_i^{(t)}} + \mathcal{H}(\pi_i^{(t+1)}) - \mathcal{H}(\pi_i^{(t)})}\notag\\
    &= \prn{\frac{1}{\eta} - \tau} \KL{\pi^{(t+1)}_i}{\pi^{(t)}_i}  + \frac{1}{\eta}\KL{\pi^{(t)}_i}{\pi^{(t+1)}_i}.
    \label{eq:stationary}
\end{align}

\item For the third term, according to \eqref{eq:r_tilde_diff}, we have
\begin{equation*}
    \left|\widetilde{r}_i^{(t)}(a) - r_i^{(t)}(a)\right|
    \le 2 d_{TV}(\widetilde{\pi}_{-i}^{(t)}, \pi_{-i}^{(t)}).
\end{equation*}
Hence,
\begin{align*}
    \sum_{i=1}^N\big|\innprod{\widetilde{r}_i^{(t)} - r_i^{(t)}, \pi_i^{(t+1)} - \pi_i^{(t)}}\big| 
    &\le 2 \sum_{i=1}^Nd_{TV}(\widetilde{\pi}_{-i}^{(t)}, \pi_{-i}^{(t)}) \norm{\pi^{(t+1)}_i - \pi^{(t)}_i}_1\\
    &\overset{\mathrm{(i)}}{\le} \frac{2}{\sqrt{N}}\sum_{i=1}^N d_{TV}(\pi_{-i}^{(t)}, \widetilde{\pi}_{-i}^{(t)})^2 + \frac{\sqrt{N}}{2}\sum_{i=1}^N \norm{\pi^{(t+1)}_i - \pi^{(t)}_i}_1^2\\
    &\overset{\mathrm{(ii)}}{\le} \frac{1}{\sqrt{N}}  \sum_{i=1}^N \KL{{\pi}_{-i}^{(t)}}{\widetilde{\pi}_{-i}^{(t)}} + \sqrt{N}\sum_{i=1}^N \KL{\pi^{(t+1)}_i}{\pi^{(t)}_i}\\
    &\le\frac{1}{\sqrt{N}}  \sum_{i=1}^N \KL{{\pi}^{(t)}}{\pi^{(t+1)}} + \sqrt{N} \KL{\pi^{(t+1)}}{\pi^{(t)}}\\
    &= \sqrt{N} \prn{\KL{\pi^{(t+1)}}{\pi^{(t)}} + \KL{\pi^{(t)}}{\pi^{(t+1)}}} = \sqrt{N} J(\pi^{(t+1)}, \pi^{(t)}),
\end{align*}
where (i) results from Young's inequality and (ii) is due to Pinsker's inequality.

\end{itemize}

Combining all pieces together, we have
\begin{align*}
    \Phi_\tau^{(t+1)} - \Phi_\tau^{(t)} 
    &\ge \prn{\frac{1}{\eta}-\tau}\sum_{i=1}^N\left[ \KL{\pi^{(t+1)}_i}{\pi^{(t)}_i}  + \KL{\pi^{(t)}_i}{\pi^{(t+1)}_i}\right]- \sqrt{N}J(\pi^{(t+1)}, \pi^{(t)})\\
    &\ge \prn{\frac{1}{\eta} - \sqrt{N} - \tau}J(\pi^{(t+1)}, \pi^{(t)}).
\end{align*}


\subsection{Proof of \eqref{eq:perf_impv_part2} }

Alternatively, we can decompose $\Phi_\tau^{(t+1)} - \Phi_\tau^{(t)}$ as
\begin{align*}
    &\Phi_\tau^{(t+1)} - \Phi_\tau^{(t)}\\
    &=\Phi^{(t+1)} - \Phi^{(t)} + \tau \sum_{i=1}^N \brk{ \cH(\pi_i^{(t+1)}) - \cH(\pi_i^{(t)})}\\
    &= \sum_{i=1}^{N} \brk{\Phi(\pi_i^{(t+1)}, \pi_{-i}^{(t)}) - \Phi^{(t)} + \tau \cH(\pi_i^{(t+1)}) - \tau \cH(\pi_i^{(t)})} + \Phi^{(t+1)} - \Phi^{(t)} - \sum_{i=1}^{N} \brk{\Phi(\pi_i^{(t+1)}, \pi_{-i}^{(t)}) - \Phi^{(t)}}\\
    &= \sum_{i=1}^{N} \brk{u_i(\pi_i^{(t+1)}, \pi_{-i}^{(t)}) - u_i(\pi^{(t)}) + \tau \cH(\pi_i^{(t+1)}) - \tau \cH(\pi_i^{(t)})} + \Phi^{(t+1)} - \Phi^{(t)} - \sum_{i=1}^{N} \brk{\Phi(\pi_i^{(t+1)}, \pi_{-i}^{(t)}) - \Phi^{(t)}}.
\end{align*}
The first term is lower bounded by $(1/\eta - \tau) J(\pi_i^{(t+1)}, \pi_i^{(t)})$ as shown in \eqref{eq:stationary}. For the remaining terms, we have
\begin{align}
    &\Big|\Phi^{(t+1)} - \Phi^{(t)} - \sum_{i=1}^{N} \brk{\Phi(\pi_i^{(t+1)}, \pi_{-i}^{(t)}) - \Phi^{(t)}}\Big|\notag\\
    &\le \sum_{\bm{a} \in \cA^N} \Phi(\bm{a}) \pi^{(t)}(\bm{a})\Bigg|\frac{\pi^{(t+1)}(\bm{a})}{\pi^{(t)}(\bm{a})} - 1 - \sum_{i=1}^N\prn{\frac{\pi_i^{(t+1)}(a_i)}{\pi_i^{(t)}(a_i)} - 1}\Bigg|.
    \label{eq:non_stationary}
\end{align}
To continue, we need the following elementary lemma, which will be proved at the end.
\begin{lemma}
    \label{lem:log_magic}
    For all $x \in (-1, \infty)$, it holds that
    \[
        0 \le x - \log(1+x) \le x \log(1+x).
    \]
\end{lemma}

Invoking Lemma \ref{lem:log_magic} to obtain
\begin{align*}
    &\Bigg|\frac{\pi^{(t+1)}(\bm{a})}{\pi^{(t)}(\bm{a})} - 1 - \sum_{i=1}^N\prn{\frac{\pi_i^{(t+1)}(a_i)}{\pi_i^{(t)}(a_i)} - 1}\Bigg|\\
    &= \Bigg|\frac{\pi^{(t+1)}(\bm{a})}{\pi^{(t)}(\bm{a})} - 1 - \log \frac{\pi^{(t+1)}(\bm{a})}{\pi^{(t)}(\bm{a})} - \sum_{i=1}^N\prn{\frac{\pi_i^{(t+1)}(a_i)}{\pi_i^{(t)}(a_i)} - 1 - \log\frac{\pi_i^{(t+1)}(a_i)}{\pi_i^{(t)}(a_i)}}\Bigg|\\
    &\le \prn{\frac{\pi^{(t+1)}(\bm{a})}{\pi^{(t)}(\bm{a})} - 1}\log \frac{\pi^{(t+1)}(\bm{a})}{\pi^{(t)}(\bm{a})} + \sum_{i=1}^N\prn{\frac{\pi_i^{(t+1)}(a_i)}{\pi_i^{(t)}(a_i)} - 1} \log \frac{\pi_i^{(t+1)}(a_i)}{\pi_i^{(t)}(a_i)}.
\end{align*}
Plugging the above inequality into \eqref{eq:non_stationary} yields
\begin{align*}
    &\Big|\Phi^{(t+1)} - \Phi^{(t)} - \sum_{i=1}^{N} \brk{\Phi(\pi_i^{(t+1)}, \pi_{-i}^{(t)}) - \Phi^{(t)}}\Big|\notag\\
    &\le \Phi_{\rm max} \sum_{\bm{a} \in \cA^N} \pi^{(t)}(\bm{a}) \brk{\prn{\frac{\pi^{(t+1)}(\bm{a})}{\pi^{(t)}(\bm{a})} - 1}\log \frac{\pi^{(t+1)}(\bm{a})}{\pi^{(t)}(\bm{a})} + \sum_{i=1}^N\prn{\frac{\pi_i^{(t+1)}(a_i)}{\pi_i^{(t)}(a_i)} - 1} \log \frac{\pi_i^{(t+1)}(a_i)}{\pi_i^{(t)}(a_i)}}\\
    &= \Phi_{\rm max} \sum_{\bm{a} \in \cA^N} \brk{\prn{\pi^{(t+1)}(\bm{a}) - \pi^{(t)}(\bm{a})}\log \frac{\pi^{(t+1)}(\bm{a})}{\pi^{(t)}(\bm{a})} } + \Phi_{\rm max} \sum_{i=1}^N \sum_{a_i \in \cA}\prn{\pi_i^{(t+1)}(a_i) - \pi_i^{(t)}(a_i)} \log \frac{\pi_i^{(t+1)}(a_i)}{\pi_i^{(t)}(a_i)}\\
    &= \Phi_{\rm max} \prn{J(\pi^{(t+1)}, \pi^{(t)}) + \sum_{i=1}^N J(\pi_i^{(t+1)}, \pi_i^{(t)})} = 2 \Phi_{\rm max} J(\pi^{(t+1)}, \pi^{(t)}).
    \label{eq:non_stationary}
\end{align*}
Combining all pieces together, we have
\begin{align*}
    &\Phi_\tau^{(t+1)} - \Phi_\tau^{(t)} \ge \prn{\frac{1}{\eta} - 2\Phi_{\rm max} - \tau}\sum_{i=1}^N J(\pi^{(t+1)}, \pi^{(t)}).
\end{align*}

\vspace*{1ex}
\begin{proof}[Proof of Lemma~\ref{lem:log_magic}]
    We have $x - \log(1+x) = x \log(1+x) = 0$ and $(x - \log(1+x))' = (x \log(1+x))' = 0$ when $x = 0$. It follows that $x - \log(1+x) \ge 0$ since $\log$ is concave. With straightforward calculation, we get
    \begin{align*}
        (x \log(1+x))' - (x - \log(1+x))'
        &= \log (1+x) 
        \begin{cases}
            \ge 0 & x \ge 0\\
            < 0 & -1 < x < 0
        \end{cases},
    \end{align*}
    which implies $x - \log(1+x) \le x \log(1+x)$.
\end{proof}

\subsection{Proof of Corollary~\ref{corollary:rate_QRE}}
\label{proof:rate_QRE}

By noting that $\pi_i^{\star(0)}\propto \exp\left(r_{i}^{(0)}/\tau\right)$, and with uniform policy initialization  $\pi_i^{(0)}\propto 1$, we can conclude
$$\norm{\log \pi_i^{(0)} - \log \pi_i^{\star(0)}}_\infty\le 2\norm{\frac{r_i^{(t)}}{\tau}-0}_{\infty}\leq \frac{2}{\tau}.$$
where the first inequality follows from Lemma \ref{lem:log_pi_gap}, and the second inequality is true since the payoff is bounded by $1$.
On the other hand, we have
\begin{align*}
  \Phi_\tau^{(T)} - \Phi_\tau^{(0)} &=\Phi (\pi^{(T)})  - \Phi(\pi^{(0)})  + \tau \mathcal{H}(\pi^{(T)}) - \tau \mathcal{H}(\pi^{(0)})\\
    &\le \Phi(\pi^{(T)})  - \Phi(\pi^{(0)})  \le \Phi_{\max},
\end{align*}  
where the first inequality uses the fact that the entropy is maximized for uniform policies, and the second inequality uses $0\leq \Phi(\pi)\leq \Phi_{\max}$ for any $\pi$. Combining the above two bounds with Theorem~\ref{thm:QRE_convergence}, we have
\begin{align*}
    \frac{1}{T}\sum_{t=1}^{T} \texttt{QRE-gap}_{\tau}^{(t)}  \leq \frac{4}{\tau\eta T} + \frac{2}{\tau} \sqrt{\frac{2\Phi_{\max}}{\eta T}}.
\end{align*}
 Setting $\eta = \frac{1}{2(\min\{\sqrt{N}, 2\Phi_{\rm max}\}+\tau)}$ and $T = \mathcal{\mathcal{O}}\prn{\frac{\min\{\sqrt{N}, \Phi_{\rm max}\}\Phi_{\max}}{\tau^2\epsilon^2}}$ thus completes the proof.